\documentclass[10pt]{amsart}
\usepackage{amsfonts,amssymb,amscd,amsmath,enumerate,verbatim,calc}
\usepackage{mathrsfs}
\numberwithin{equation}{section}
\newtheorem{theorem}{Theorem}[section]
\newtheorem{corollary}[theorem]{Corollary}
\newtheorem{lemma}[theorem]{Lemma}

\newtheorem{proposition}[theorem]{Proposition}
\newtheorem{definition}[theorem]{Definition}

\newtheorem{example}[theorem]{Example}
\newtheorem{remark}[theorem]{Remark}

\newcommand{\bC}{\mathbb{C}}

 \DeclareMathOperator{\q}{q}

\begin{document}

\title[Riemannian geometry on Hom-$\rho$-commutative algebras]
 {Riemannian geometry on Hom-$\rho$-commutative algebras}

\bibliographystyle{amsplain}

\author[Zahra Bagheri and Esmaeil Peyghan]{Z. Bagheri and  E. Peyghan}
\address{Department of Mathematics, Faculty of Science, Arak University,
Arak, 38156-8-8349, Iran.}
\email{z-bagheri@phd.araku.ac.ir, e-peyghan@araku.ac.ir}





\begin{abstract}
Recently, some concepts such as Hom-algebras, Hom-Lie algebras, Hom-Lie admissible algebras, Hom-coalgebras are studied and some classical properties of algebras and some geometric objects are extended on them. In this paper by recall the concept of Hom-$\rho$-commutative algebras, we intend to develop some of the most classical results in Riemannian geometry such as metric, connection, torsion tensor, curvature tensor on it and also we discuss about differential operators and get some results of differential calculus by using them. The notions of symplectic structures and Poisson structures are included and an example of $\rho$-Poisson bracket is given.
\end{abstract}

\maketitle







\section{Introduction}
One branch of differential geometry is Riemannian geometry that studies Riemannian manifolds (a smooth manifold with a Riemannian metric) (see \cite{S t} for more details). Riemannian geometry at the first was
brought up by Bernhard Riemann in the nineteenth century. The concept of linear connection is one of the
main concepts of the Riemannian geometry, which arose by the idea of parallel transport along a path in a Riemannian manifold at the end of 19th century. (see \cite{S H}, \cite{R W}). There is no direct and quick way to
companion between distance points of a curve space,  however the
connection permits to contrast what is happening at these points.
Earlier, in the 1910's, Albert Einstein discovered that the Riemannian geometry is substantial to general relativity theory. It is also the foundational revelation for gauge theories. This division into two branches has led to many representations tending to either the specific (e.g presented in tensor notation assuming a coordinate frame and zero torsion) or
the abstract (e.g using the language of fiber bundles). By worth of its applications the Riemannian geometry stands at the nucleus of modern mathematics.

Differential calculus is a branch of mathematics concerned with the determination, properties, and application of derivatives and differentials in study of functions. The development of differential calculus is closely dealing with  the concept of integral calculus. In this approach the differential calculus on the manifold is deduced from the properties of the manifold and it involves functions on the manifold, differential operators,
differential forms and derivatives. If we denote by $d$ the exterior derivative on differential forms, then the operator $d$ satisfies $d^2 = 0$ that is an important property of this calculus.

Definition of a multiplication over a vector space was the origin of the notion of Hom-algebra structure, where the structure twisting by a homomorphism. The structure of Hom-Lie algebra appeared first as a generalization of Lie algebra by Hartwig, Larsson and Silvestrov in \cite{J D S}. Physics and deformations of Lie algebras, in particular Lie algebra of vector fields were the stimulants to study Hom-Lie structures. Lie algebras are special cases of Hom-Lie algebras in which $\phi$ is the identity map. Also, $q$-deformations of the Witt and the Virasoro algebras have the structure of a Hom-Lie algebra. Later Hom-Lie algebras were extended to Hom-associative algebras by Makhlouf and Silverstrov in \cite{A S} and to quasi-Hom Lie and quasi-Lie algebras by Larsson and Silvestrov in \cite{L S}, \cite{L S S}. Leibniz and Hom-Lie admissible algebras, Hom-alternative algebras, Hom-Hopf algebras, Hom-coalgebras  are other interesting Hom-algebraic structures were studied \cite{M, M S, Y}.

Non-commutative geometry is a branch of mathematics concerned with a geometric approach to non-commutative algebras, and with the construction of spaces that are presented by non-commutative algebras of functions. Extension of the
concept of differential forms on manifolds plays the basic role in non-commutative geometry (see \cite{C2}, \cite{D1}, \cite{M2}, \cite{NML}, for instance). Important examples of non-commutative geometry are  $\rho$-commutative algebras. They have a great ability to generalize geometric objects. Accordingly, Riemannian geometry and its objects such as metric, connection, curvature, torsion, differential form and also differential calculus and application to hyperplane are discussed on $\rho$-commutative algebra by Bongaarts, Ciupala and Ngakeu in \cite{C}, \cite{BP1} and \cite{N}.
In this paper we recall and study Hom-$\rho$-commutative algebra and
develop some of the most classical results in Riemannian geometry and differential calculus on it.

This paper is arranged as follows. In Section 2, we recall some necessary background knowledge including $\rho$-commutative and
Hom-$\rho$-commutative algebras, Hom-associative and Hom-$\rho$-commutative Lie algebras. In Section 3, we define $p$-forms and wedge product on Hom-$\rho$-algebras and the reader will get some important properties of $\rho$-tensor products, this section further develops the foundational topics for Riemannian manifolds, metric, connection, torsion tensor and curvature tensor are included. Also, we check some examples, properties and lemmas to obtain important results. Section 4 has been assigned to discuss about representations, cochain, Hom-cochain and some results will be derived of differential calculus. Also, symplectic structures and Poisson brackets are studied in this section.


\section{Hom-$\rho$-commutative algebra}
In this section, we summarize some definitions concerning $\rho$-commutative and Hom-$\rho$-commutative algebras
and related results.

Let $A$ be an associative and unital algebra over a field $k$ ($k = \mathbb{R}$ or $k = \mathbb{C}$), grading by an abelian group $(G, +)$ that is  the vector space $A$ has a $G$-grading $A = \oplus_{a\in G}A_a$
such that $A_aA_b\subset A_{a+b}$. A map
$\rho:G\times G\longrightarrow k^{\star}$ is called a two-cycle if the following conditions hold
\begin{align}
&\rho(a,b) =\rho(b,a)^{-1},\quad a,b\in G,\\
&\rho(a+b,c) =\rho(a,c)\rho(b,c),\quad a,b,c\in G.
\end{align}
The above conditions say that 
$\rho(a,b)\neq 0$, $\rho(0,b)=1$ and $\rho(c,c)=\pm 1$ for all $a,b,c\in A$ , $c\neq 0$.

The $\rho$-commutator of two homogeneous elements $f$ and $g$ of $A$ is
\begin{equation}\label{111}
[f, g]_{\rho} = fg-\rho(|f|, |g|)gf,
\end{equation}
where $|f|$ is the $G$-degree of a (non-zero)  homogeneous element $f\in A$ and the set of homogeneous elements in $A$ is denoted by
 $Hg(A)$.\\
A $\rho$-commutative algebra is a $G$-graded algebra $A$ with a given two-cycle
$\rho$ such that $f g =\rho(| f |, |g|)g f $ for all homogeneous elements $ f$ and
$g$ in $A$ (i.e., $[f,g]_{\rho}=0$).\\
In the following, we have some preliminary definitions from \cite{AAM, I B}:
\begin{definition}
A Hom-$\rho$-algebra is a quadruple
$ (A, \cdot, \rho,\phi)$ 
consisting of a  $G$-graded
vector space
 $A$,
 i.e., $A =\bigoplus_{a\in G} A_a$,
 an even bilinear map
 $ \cdot : A \times A \longrightarrow A$ 
 i.e.,
 $A_a\cdot A_b \subseteq A_{a+b}$, 
for all 
$a,b \in G$,
a two-cycle
 $\rho : G\times G\longrightarrow k^{\ast}$
 and an even linear map
$\phi : A\longrightarrow A$.
 In addition if
 $f \cdot g=\rho (\vert f\vert ,\vert g\vert)g\cdot f$,
  for any
 $f,g\in Hg(A)$, we have a Hom-$\rho$-commutative algebra.
\end{definition}
A Hom-$\rho$-algebra $(A,\cdot,\rho,\phi)$ is called a Hom-associative $\rho$-algebra if
$$\phi(f)(g\cdot h) = (f\cdot g) \phi(h).$$
Moreover if $fg=\rho(|f|,|g|)gf$, $(A,\cdot,\rho,\phi)$ is called Hom-associative $\rho$-commutative algebra.\\
A Hom-$\rho$-algebra  $(A, \cdot,\rho,\phi)$ is said to be multiplicative if $\phi$ is a morphism for $\cdot$, regular if $\phi$ is an automorphism for $\cdot$, and involutive if $\phi^2 = Id_A$ (see \cite{AAM,I B}).
\begin{example}\label{13}
The quaternion algebra $H$ is a $\mathbb{Z}_2\times\mathbb{Z}_2\times\mathbb{Z}_2$-graded algebra in the following sense. Associate the " Triple degree" to the standard basis elements of H
$$\varepsilon=(0,0,0), i=(0,1,1), j=(1,0,1),k=(1,1,0),$$
where $\varepsilon$ denotes the unit and the following multiplication conditions are imposed
\begin{enumerate}
\item[i)]
$i^2 = j^2 = k^2 = -1,$
\item[ii)]
$ij = k, ji = -k, jk = i, kj = -i, ki = j, ki = j, ik = -j.$
\end{enumerate}
Also the two-cocycle $\rho$ is defined by $\rho(a,b)=(-1)^{\langle a,b\rangle}$, where $\langle a,b\rangle$ is the usual scalar product of 3-vectors. Indeed $\langle i,j\rangle=1$ and similarly for $k$, so that $i,j$ and $k$, $\rho$-commute with each other. But $\langle i,i\rangle=\langle j,j\rangle=\langle k,k\rangle=2$, so that $i,j$ and $k$ commute with themselves. Thus, quaternion algebra $H$ is a $\rho$-commutative algebra. 
If we set linear map $\phi_H(i)=ai,\phi_H(j)=bj,\phi_H(k)=ck,\quad a,b,c\in\bC$, then we have a Hom-$\rho$-commutative quaternion algebra. But, $H$ is a Hom-associative $\rho$-commutative algebra if $a=b=c$.
\end{example}
\begin{definition}
A Hom-$\rho$-Lie algebra is a $G$-graded vector space $A$ together with a bilinear map $[\cdot,\cdot]_{\rho}:A\times A\longrightarrow A$, a two-cycle $\rho$ and a linear map $\phi:A\longrightarrow A$
satisfying the following relations
\begin{align*}
&\bullet [A_g,A_{g^{\prime}} ]_{\rho}\subset A_{g+g^{\prime}},\\
&\bullet [f, g]_{\rho} = -\rho(|f|, |g|)[g, f]_{\rho},\\
&\bullet \rho(|h|, |f|)[\phi(f), [g, h]_{\rho}]_{\rho}+\rho(|g|, |h|)[\phi(h), [f, g]_{\rho}]_{\rho}+\rho(|f|, |g|)[\phi(g), [h, f]_{\rho}]_{\rho} = 0.
\end{align*}
The second condition is called $\rho$-antisymmetry and the third condition is called $\rho$-Jacobi identity.
\end{definition}
\begin{proposition}\cite{LY}
The multiplex $(A,\cdot,\rho,\phi,[\cdot,\cdot]_{\rho})$ consisting of a Hom-associative $\rho$-algebra $(A,\cdot,\rho,\phi)$ and $\rho$-commutator $[f,g]_{\rho}=fg-\rho(|f|,|g|)gf$ is a  Hom-$\rho$-Lie algebra that is called  Hom-$\rho$-commutative Lie algebra.
 \end{proposition}
A Hom-$\rho$-Lie algebra $(A,\rho,\phi,[\cdot,\cdot]_{\rho})$ is called multiplicative if $\phi[f,g]_{\rho}=[\phi(f),\phi(g)]_{\rho}$, regular if $\phi$ is an automorphism and involutive if $\phi$ is an involution, that is $\phi^2 = Id_A$ (see \cite{AAM}).
\begin{definition}\cite{I B}
Let $(A,\cdot,\rho,\phi)$ be a Hom-$\rho$-algebra. A $\rho$-derivation of degree $|X|$ on $A$ is a linear map $X : A \longrightarrow A$ 
such that
\begin{equation}
X(f\cdot g) = X(f)\cdot g + \rho(|X|, |f|) f\cdot X(g).
\end{equation}
\end{definition}
If $\rho\text{-}DerA$ is denoted the space of all $\rho$-derivations of $A$,   then for $X\in Hg(\rho\text{-}DerA)$ and $Y\in Hg(\rho\text{-}DerA)$
the $\rho$-commutator of $X,Y$, defined by $[X, Y ]_{\rho} = X\circ Y - \rho(|X|, |Y |)Y\circ X$, is
a $\rho$-derivation. Furthermore, when $A$ is Hom-$\rho$-commutative, $\rho$-$DerA$ is also a
$A$-bimodule with actions $\rhd,\lhd$ defined by
\begin{equation}\label{1}
(f\rhd X).g = f(X.g), X \lhd f = \rho(|X|, |f|)f\rhd X.
\end{equation}
In fact any $G$-graded left module $M$ over a Hom-$\rho$-commutative algebra $A$ is a $A$-bimodule
with
\begin{equation}
f\rhd (X\lhd g) = (f\rhd X)\lhd g,\quad  f, g \in A,X \in M.
\end{equation}
Moreover, $\rho\text{-}DerA$ equipped with the $\rho$-commutator is a $\rho$-Lie algebra.
 
We will simply denote the left action $ f\rhd X$  as $fX$ and
$\rho(\vert X\vert ,\vert Y\vert)$ by 
$\rho(X,Y).$
\section{$p$-forms, $\rho$-Tensor product}
In this section, at first we define the $p$-forms and the wedge product and the next we turn to a brief discussion of the concept of tensor products to develop  Riemannian geometry on Hom-$\rho$-commutative algebras, then we concentrate on metrics to introduce the notion of linear connection and define the torsion and curvature associated to the connection. Later, we survey some properties and give to the reader some main points. Symbols 
$Hg(\Omega^1(A))$ 
and 
$Hg(\rho$-$DerA)$ 
apply respectively 
for homogeneous elements of 
$\Omega^1(A)$ and $\rho$-$DerA$.
\begin{definition}
A $p$-form on Hom-$\rho$-algebra $A$ is a $p$-linear map $\alpha_p : \times^p (\rho\text{-}DerA) \longrightarrow A$, $p$-linear in sense of left $A$-modules
\begin{align*}
&\bullet\alpha_p (f X_1, \cdots ,X_p) = f \alpha_p (X_1,\cdots  ,X_p),\\
&\bullet\alpha_p (X_1, \cdots ,X_j f,X_{j+1},\cdots  ,X_p) =\alpha_p (X_1,\cdots , X_j , f X_{j+1}, \cdots , X_p),\\
&\bullet\alpha_p (X_1, \cdots ,X_j ,X_{j+1}, \cdots ,X_p)=-\rho(X_j  ,  X_{j+1})\alpha_p (X_1, \cdots , X_{j+1}, X_j, \cdots ,X_p),
\end{align*}
where $j = 1, \cdots ,p-1$, $X_k\in Hg(\rho\text{-}DerA)$, $k = 1, \cdots ,p$, $f\in A $ and $ Xf$ is the right $A$-action
on $\rho$-$DerA$ defined by \eqref{1}. Let $\Omega^p(A)$ denotes the set of $p$-forms, then
$\Omega^p (A)$ is a $G$-graded $A$-bimodule with
\begin{align*}
&(\alpha_pf)(X_1, \cdots ,X_p) = \alpha_p(X_1, \cdots ,X_p)f,\\
&\ \ \ \ f\alpha_p=\rho(f,\alpha_p)\alpha_pf,\\
&|\alpha_p| = |\alpha_p (X_1, \cdots ,X_p)| - (|X_1| +\cdots + |X_p|).
\end{align*}
Then we have the exterior algebra
$\Omega(A)=\bigoplus_{p\geq0}\Omega^p(A)$ with $\Omega^0(A) = A$.
\end{definition}
\begin{definition}
The wedge product $\wedge $ in $\Omega(A)$ is the map $$\wedge:\Omega^p(A)\times\Omega^q(A)\longrightarrow \Omega^{p+q}(A),$$ 
defined by
\begin{align*}	&(\alpha\wedge\beta)(X_1,\cdots,X_p,\cdots,X_{p+q})\\
&=\sum_{\sigma\in S_{p,q}}sign(\sigma) \times \rho(\displaystyle\sum_{j=p+1}^{p+q} X_{\sigma(j)}, \alpha)\rho(X_{\sigma(k)}, X_{\sigma(l)})
\times \alpha( X_{\sigma(1)},\cdots ,X_{\sigma(p)})\beta( X_{\sigma(p+1)}, \cdots,X_{\sigma(p+q)}),
\end{align*}
for $\alpha\in Hg(\Omega^p(A)$ and $\beta\in Hg(\Omega^q(A)$, where 
$S_{p,q}$ is the set of permutations $\sigma\in S_{p+q}$ such
that $\sigma(1) < \sigma(2) < \cdots < \sigma(p)$ and $\sigma(p + 1) < \sigma(k + 2) < \cdots < \sigma(p+q)$, $ l<k$ and $\sigma(l) > \sigma(k)$.
\end{definition}
\textbf{$\rho$-tenosor product: }
Let $(A, \cdot,\rho, \phi)$ be a Hom-$\rho$-commutative algebra and $(\rho\text{-}DerA,[\cdot,\cdot]_{\rho},\rho,\phi_A)$ be a Hom-$\rho$-Lie algebra, where $\rho\text{-}DerA$ is the set of all $\rho$-derivation on Hom-$\rho$-commutative algebra $A$.\\

 For 
 $\alpha_1,\cdots ,\alpha_p\in Hg(\Omega^1(A))$,
we make a $p$-linear homogeneous map $s$ of  $G$-degree 
$|s| =\sum^p_{i=1}|\alpha_i\vert$ as
$s=\alpha_1\otimes_{\rho}\cdots\otimes_{\rho}\alpha_p$
and define it by
\begin{align*}
\alpha_1\otimes_{\rho}\cdots\otimes_{\rho}\alpha_p( X_1, \cdots ,X_p): &=\prod_{i=1}^p\alpha_i(X_i)\prod_{k=1}^{p-1}\rho(\sum^p_{j=k+1}
X_j, \alpha_k),
\end{align*}
where
$X_1,\cdots ,X_p\in Hg (\rho\text{-}DerA)$, 
that satisfy the following relations
\begin{align*}
&s( f X_1,X_2,\cdots ,X_p)= f s(X_1,X_2,\cdots ,X_p),\\
&s(X_1, \cdots,X_i \lhd f,X_{i+1}, \cdots ,X_p)=s(X_1, ...X_i, fX_{i+1}, \cdots ,X_p).
\end{align*}
Let $T^{\otimes^p_{\rho}}$ denotes the linear space generated by elements
 $s= \alpha_1\otimes_{\rho}\cdots\otimes_{\rho}\alpha_p$. $T^{\otimes^p_{\rho}}$ is a $A$-bimodule with the following actions
 \begin{align*}
 &(\alpha_1\otimes_{\rho}\cdots\otimes_{\rho}\alpha_p)f=\alpha_1\otimes_{\rho}\cdots\otimes_{\rho}(\alpha_pf),\\
 &f(\alpha_1\otimes_{\rho}\cdots\otimes_{\rho}\alpha_p)=(f\alpha_1)\otimes_{\rho}\cdots\otimes_{\rho}\alpha_p=\rho(f,\sum_{i=1}^p\alpha_i)\alpha_1\otimes_{\rho}\cdots\otimes_{\rho}(\alpha_pf).
 \end{align*}
We have the $\rho$-tensor algebra
$T^{\otimes}=\bigoplus_{p\geqslant 0}T^{\otimes^p_{\rho}}$ with 
$T^{\otimes^0_{\rho}}=A$ and natural algebra structure $\otimes_{\rho}$, which is defined on homogeneous elements in $T^{\otimes}$ by
\begin{align*}
&T_p\otimes_{\rho}T_q(X_1,X_2, \cdots , X_p,X_{p+1}, \cdots,X_{p+q})=\\
& \ \ \ \ \ \ \ \ \ \ \ \ \ \ \ \ \ \ \ \ T_p(X_1, X_2,\cdots, X_p)T_q(X_{p+1},\cdots , X_{p+q})\times\rho(\sum_{j=1}^q X_{p+j}, T_p),
\end{align*}
$\forall T_p\in T^{\otimes^p_{\rho}},T_q\in T^{\otimes^q_{\rho}}, X_1,\cdots ,X_{p+q}\in Hg(\rho\text{-}DerA)$.
Also, $T_p\otimes_{\rho}T_q$ has the degree 
$ |T_p\otimes_{\rho}T_q| = |T_p|+|T_q|$. Notice that
$T^{\otimes^p_{\rho}}$
 coincide with the space of all 
$\rho$-$p$-linear
 maps on
 $\times^p\rho\text{-}DerA$, if
 $\Omega^1(A)$
 and
 $\rho$-$Der_{\phi}A$
 are finitely generated. In this case,
$T\in T^{\otimes^p_{\rho}}$
is called {\it covariant 
$\rho$-tensor}.
\begin{definition}
A covariant 2-tensor $g\in T^{\otimes^2_{\rho}}$ is called non-degenerate tensor if 
$$\forall\alpha\in Hg(\Omega^1(A)), \exists X_{\alpha}\in\rho\text{-}DerA :~~ g(Y,X_{\alpha}) =\alpha(Y),\quad\forall Y\in\rho\text{-}DerA.$$
A (homogeneous) metric $g$ of degree $| g|$ on a $\rho$-commutative algebra $A$ is an homogeneous
symmetric and nondegenerate covariant 2-tensor on $A$.
\end{definition}
\begin{definition}
Let $(A, \cdot,\rho, \phi)$ be a multiplicative Hom-$\rho$-commutative algebra and $(\rho\text{-}DerA,[\cdot,\cdot]_{\rho},\rho,\phi_A)$ be a Hom-$\rho$-Lie algebra, where $\rho\text{-}DerA$ is the set of all $\rho$-derivations on $A$ endowed with a bilinear map $g(.,.)$ such that for any $ X,Y\in Hg(\rho\text{-}Der(A))$ the following equations are satisfied
\begin{align*}
&(i)~~g(fX,Y)=fg(X,Y),\quad \forall X,Y\in \rho\text{-}Der(A),\\
&(ii)~~g(X,Y)=\rho(X,Y)g(Y,X),\quad \forall X,Y\in Hg(\rho\text{-}Der(A)),\\
&(iii)~~g(X,Y) =g(\phi_A(X),\phi_A(Y)),\quad \forall X,Y\in \rho\text{-}Der(A)\\
&(iv)~~\text{The map}~~\tilde{g}:\rho\text{-}DerA\longrightarrow \Omega^1(A) \text{defined by}~~\tilde{g}(X)Y=g(Y,X) ~~\text{is a left}~~ A\text{-module isomorphism}.
\end{align*}
Then, we say that $A$ admits a  metric $g(.,.)$ (not necessary homogeneous). Also, homogeneous metric $\tilde{g}$  satisfies $g(X, fY)= g(Xf, Y)$ and $\tilde{g}(aX) =
a\tilde{g}(X)$, $\tilde{g}(X\lhd a) = \tilde{g}(X)a \rho(a,g), a\in Hg(A).$
\end{definition}
\begin{definition}
A linear connection on $\rho$-$DerA$ is a linear map
\[
\begin{cases}
\nabla : \rho\text{-}DerA \longrightarrow End(\rho\text{-}DerA),\\
 \ \ \ \ \ \ \ \ X \longrightarrow\nabla_X,
 \end{cases}
 \]
such that
\begin{align*}
&\nabla_{aX}Y  = a\nabla_XY, \quad a\in A,X, Y\in \rho\text{-}DerA,\\
&\nabla_X(aY ) = (X\cdot a)Y + \rho(X, a)\phi(a)\nabla_XY,\quad a \in Hg(A), X \in Hg(\rho\text{-}DerA).
\end{align*}
\end{definition}
\begin{definition}
Let $T \in T^{\otimes}$, $X \in Hg(\rho\text{-}DerA)$ and $\nabla_XT$ denotes the covariant derivative of
$T$, which is a $p$-linear map. This map is defined on $\rho\text{-}DerA$ by
\begin{align*}
\rho(X,\sum^p_{i=1}X_i)\nabla_XT(X_1, \cdots , X_p) &=\phi_A(X)\cdot T(X_1, \cdots, X_p)\\
& -\sum^p_{i=1}\rho(X,\sum^{i-1}_{l=1}X_l)T(\phi_A(X_1),\cdots ,\nabla_XX_i, \cdots , \phi_A(X_p)).
\end{align*}
$\rho$-tensor $T$ is said to be “parallel” or compatible with respect to a linear connection $\nabla$ if $\nabla T=0$. 
\end{definition}
Connection $\nabla$ is called {\it compatible with metric} 
$g$
if 
$$ \phi_A(X).g(Y, Z) =g(\nabla_XY,\phi_A(Z)) +\rho(X,Y)g(\phi_A(Y),\nabla_XZ),\quad\forall X,Y,Z\in Hg(\rho\text{-}DerA).$$
The curvature of a linear connection $\nabla$ is the map
\[
\begin{cases}
R : \rho\text{-}DerA \times\rho\text{-}DerA\longrightarrow End(\rho\text{-}DerA),\\
\ \ \ \ \ \ \ \ (X, Y ) \longrightarrow R(X, Y ):=R_{XY},
\end{cases}
\]
defined by
\[
R(X, Y )(Z) = \nabla_{\phi_A(X)}\nabla_YZ- \rho(X,Y )\nabla_{\phi_A(Y)} \nabla_XZ-\nabla_{[X,Y ]_{\rho}}\phi_A(Z).
\]
\begin{lemma}
 By the definition of curvature $R$, we get the following equalities
\begin{align*}
&a)\quad 
R(X, Y)= -\rho(X,Y)R(Y,X),\\
&b)\quad R(aX,Y)Z= \phi(a)R(X,Y)Z-\rho(X+a,Y)(\phi_A(Y)\cdot a)\nabla_XZ\\
&\ \ \ +\rho(X+a,Y)(Y\cdot a)\nabla_X\phi_A(Z)
+a\nabla_{\phi_A(X)}\nabla_YZ-\phi(a)\nabla_{\phi_A(X)}\nabla_YZ,\\
&c)\quad 
R(X,aY)(Z) = \rho(X,a)\phi(a)(R(X,Y) Z)+(\phi_A(X)\cdot a)\nabla_YZ\\
&\ \ \ -(X\cdot a)\nabla_Y\phi_A(Z)-\rho(X,a+Y)a\nabla_{\phi_A(Y)}\nabla_XZ+\rho(X,a+Y)\phi(a)\nabla_{\phi_A(Y)}\nabla_XZ\\
&\ \ \ -\phi(a)\nabla_{[X,Y]_{\rho}}\phi_A(Z)+\rho(X,a)\phi(a)\nabla_{[X,Y]_{\rho}}\phi_A(Z),\\
&d)\quad
R(X,Y) (aZ) = \rho(X+Y, a)\phi^2(a)(R(X,Y) Z)+(\phi_A(X)\cdot Y\cdot a)Z\\
&\ \ \ -\rho(X,Y)(\phi_A(Y)\cdot X\cdot a)Z+\rho(Y,a)(\phi_A(X)\cdot\phi(a))\nabla_YZ\\
&\ \ \ -\rho(X,Y)\rho(X,a)(\phi_A(Y)\cdot\phi(a))\nabla_XZ-\rho(X,Y)\rho(Y,X+a)\rho(X,a)\phi(a)\cdot X\\
&\ \ \ -\rho(X,Y)\rho(Y,X+a)\phi(X\cdot a)\nabla_{\phi_A(Y)}Z+\rho(X,Y+a)\phi(Y\cdot a)\nabla_{\phi_A(X)}Z\\
&\ \ \ -(X\cdot Y\cdot a)\phi_A(Z)+\rho(X,Y)(Y\cdot X\cdot a)\phi_A(Z)\\
&\ \ \ -\rho(X+Y,a)\phi(a)\nabla_{[X,Y]_{\rho}}\phi_A(Z)+\rho(X+Y,a)\phi^2(a)\nabla_{[X,Y]_{\rho}\phi_A(Z)}.
\end{align*}
\end{lemma}
\begin{proof}
For sake of the brevity, we proof the relation (b). By the definitin of curvature $R$, we have
\begin{align*}
R(aX,Y)Z&=\nabla_{\phi_A(aX)}\nabla_YZ-\rho(X+a,Y)\nabla_{\phi_A(Y)}\nabla_{aX}Z-\nabla_{[aX,Y]_{\rho}}\phi_A(Z).
\end{align*}
Now, by the properties of connection $\nabla$ and the relation $[aX,Y]_{\rho}=-\rho(X+a,Y)(Y\cdot a)X+\phi(a)[X,Y]_{\rho}$, we get
\begin{align*}
R(aX,Y)Z&=a\nabla_{\phi_A(X)}\nabla_YZ-\rho(X+a,Y)(\phi_A(Y)\cdot a)\nabla_XZ\\
&\ \ \ -\rho(X+a,Y)\rho(Y,a)\phi(a)\nabla_{\phi_A(Y)}\nabla_{X}Z\\
&\ \ \ -(Y\cdot a)\nabla_X\phi_A(Z) -\phi(a)\nabla_{[X,Y]_{\rho}}\phi_A(Z)\\
&=\phi(a)R(X,Y)Z-\rho(X+a,Y)(\phi_A(Y)\cdot a)\nabla_XZ\\
&\ \ \ +\rho(X+a,Y)(Y\cdot a)\nabla_X\phi_A(Z)
+a\nabla_{\phi_A(X)}\nabla_YZ-\phi(a)\nabla_{\phi_A(X)}\nabla_YZ.
\end{align*}
\end{proof}
Also, the torsion of the connection $\nabla$ is defined by
\[
\begin{cases}
T_{\nabla} : \rho\text{-}DerA\times\rho\text{-}DerA \longrightarrow\rho\text{-}DerA,\\
T_{\nabla}(X, Y) = \nabla_XY -\rho(X,Y)\nabla_Y X -[X, Y ]_{\rho}.&
\end{cases}
\]
Connection $\nabla$ called torsion-free if
$$T_\nabla=0,\:\text{i.e.,}\quad [X,Y]_{\rho} =\nabla_XY -\rho(X,Y)\nabla_YX,\quad\forall X,Y\in Hg(\rho\text{-}DerA).$$
In the similar way of the curvature, it is easy to see that the torsion have the following properties
\begin{align*}
T(aX,Y)&=T(X,Y)+a\nabla_XY-\phi(a)\nabla_XY,\\
T(X,aY)&=\rho(X,a)\phi(a)T(X,Y)-\rho(X,a+Y)a\nabla_YX\\
& \ \ \  +\rho(X+a,Y)\phi(a)\nabla_YX-\rho(X,a)\phi(a)[X,Y]_{\rho}+\phi(a)[X,Y]_{\rho}.
\end{align*}

\begin{theorem}
There exists a unique linear connection on every Hom-$\rho$-commutative algebra with homogeneous metric $g$, which is torsion-free and compatible with metric
$g$. This connection is called Levi-Civita connection associated to $g$.
\end{theorem}
\begin{proof}
By the compatibility condition of the connection $\nabla$ with the metric $g$, we can easily show that the Koszul equation is given by
\begin{align*}
2\rho( Z,Y)g(\phi_A(X),\nabla_YZ) &=\rho (X,Z)\phi_A(Z)\cdot g(X,Y) +\rho (X,Z) g(\phi_A(Z), [X,Y]_{\rho}) \\
&\ \ \ -\phi_A(X)\cdot g(Z,Y) -\rho (X,Z)g([Z,X]_{\rho},\phi_A(Y))\\
&\ \ \  +\rho (Z,Y)\rho(X,Y+Z) \phi_A(Y)\cdot g(Z,X) +\rho (Z,Y)\rho(X,Y+Z) g([Y,Z]_{\rho},\phi_A(X)),
\end{align*}
where $X,Y,Z\in Hg(\rho\text{-}DerA)$. So, the proof is clear.
\end{proof}
\textbf{Christoffel coefficients}: 
Let us assume that $A$ is generated as algebra by $n$ homogeneous coordinates $x_1, x_2,\cdots, x_n$ and $\rho\text{-}DerA$ is generated by $\partial_1,\cdots,\partial_n$, where $\partial_i=\frac{\partial}{\partial x_i}$ such that  $\frac{\partial}{\partial x_i}(x_j)=\delta_{i j}$, $|\partial_i|=-|x_i|$. Also, we assume that $\Omega^1(A)=\langle dx^1,\cdots,dx^n\rangle$ such that $|dx^i|=| x_i|$. We let $(dx^i)$ be the dual basis of $(\partial_i)$ and we set
	$$g = dx^m\otimes dx^n g_{mn}\quad (\text{summation}).$$
\begin{proposition}
Assuming that $[\partial_i,\partial_j]=0$ and $\nabla_{\partial_i}\partial_j=\sum_s\phi_A(\partial_s)\triangleleft \Gamma_{ij}^s$,  the $\rho$-Christoffel coefficients $\Gamma_{ij}^t$ of $G$-degree $|\Gamma_{ij}^t|=x_t-x_i-x_j$ of the Levi-Civita connection $\nabla$ are as follows
\begin{equation}\label{ZZZZ}
\Gamma_{ij}^t=\frac{1}{2}\rho(g,x_t-x_i-x_j)\times\sum_k\tilde{g}^{tk}\{-\phi_A(\partial_k)\cdot\tilde{g}_{ij} +\rho(x_i+x_k,x_j)\phi_A(\partial_j)\cdot\tilde{g}_{ki} +\rho(x_k, x_i+x_j)\phi_A(\partial_i)\cdot\tilde{g}_{jk}\}.
\end{equation}
\end{proposition}
\begin{proof}
Setting $X=
\partial_k$, $Y=\partial_i$ and $Z=\partial_j$ in the Koszul formula, we get
\begin{align*}
2g(\phi_A(\partial_k),\sum_s\phi_A(\partial_s\lhd\Gamma_{ij}^s))&=\rho(\partial_k,\partial_i+\partial_j)\phi_A(\partial_i)\cdot\tilde{g}_{jk} +\rho(\partial_k+\partial_i,\partial_j)\phi_A(\partial_j)\cdot\tilde{g}_{ki}-\phi_A(\partial_k)\cdot\tilde{g}_{ij},
\end{align*}
thus
\begin{align*}
2\sum_s\rho(\partial_s,\Gamma_{ij}^s)\rho(\partial_k,\Gamma_{ij}^s)\Gamma_{ij}^s\tilde{g}_{ks}&=\rho(\partial_k,\partial_i+\partial_j)\phi_A(\partial_i)\cdot\tilde{g}_{jk} +\rho(\partial_k+\partial_i,\partial_j)\phi_A(\partial_j)\cdot\tilde{g}_{ki}-\phi_A(\partial_k)\cdot\tilde{g}_{ij}.
\end{align*}
Multiplying both sides of the above relation by $\tilde{g}^{tk}$, we get
\begin{align*}
2\sum_s\rho(-x_t-x_k,\Gamma_{ij}^t)\rho(\tilde{g}^{tk},\Gamma_{ij}^t)\Gamma_{ij}^t&=\rho(x_k,x_i+x_j)\phi_A(\partial_i)\cdot\tilde{g}_{jk} +\rho(x_k+x_i,x_j)\phi_A(\partial_j)\cdot\tilde{g}_{ki}-\phi_A(\partial_k)\cdot\tilde{g}_{ij}.
\end{align*}
Since $g = dx^t\otimes dx^k g_{tk}$, so $|g|=x_t+x_k+g_{tk}$ and then $-x_t-x_k=g_{tk}-|g|$. Therefore
\begin{align*}
2\sum_s\rho(g_{tk},\Gamma_{ij}^t)\rho(-g,\Gamma_{ij}^t)\rho(-g_{tk},\Gamma_{ij}^t)\Gamma_{ij}^t&=\rho(x_k,x_i+x_j)\phi_A(\partial_i)\cdot\tilde{g}_{jk} +\rho(x_k+x_i,x_j)\phi_A(\partial_j)\cdot\tilde{g}_{ki}\\
&\ \ \ \ -\phi_A(\partial_k)\cdot\tilde{g}_{ij}.
\end{align*}
Thus, in the last step, the result will be obtained, that is
\begin{align*}
\Gamma_{ij}^t&=\dfrac{1}{2}\rho(g,\Gamma_{ij}^t)\sum_k\tilde{g}^{tk}\{\rho(x_k,x_i+x_j)\phi_A(\partial_i)\cdot\tilde{g}_{jk} +\rho(x_k+x_i,x_j)\phi_A(\partial_j)\cdot\tilde{g}_{ki}-\phi_A(\partial_k)\cdot\tilde{g}_{ij}\}.
\end{align*}
\end{proof}
\begin{example}\label{12}
Extended-hyperplane $A^2_q=\langle 1,x,y,x^{-1},y^{-1}:\quad x\cdot y =qy\cdot x\rangle$ is a
$\mathbb{Z}\times\mathbb{Z}$-graded algebra. Considering the two-cycle
$$\rho(n,n^{\prime}) =q^{\sum_{j,k=1}^n n_j n^{\prime}_k\alpha_{jk}},$$
where $ \alpha_{jk}=1$ if $j< k$, $0$ if 
$j=k$ and $-1$ if $j>k$, $A^2_q$ is a $\rho$-commutative algebra. In this example, we intend to define morphisms $\phi$ and $\phi_{A^2_q}$ on the $A^2_q$ and $\rho\text{-}DerA^q_2$ respectively, to make Hom-$\rho$-commutative algebra and Hom-$\rho$-commutative Lie algebra, respectively. In the next, by considering the metric $g$ and it's components that is defined in \cite{N}, we obtain the $\rho$-Christoffel coefficients $\Gamma_{ij}^t$ corresponding to $\phi_{A^2_q}$ by \eqref{ZZZZ}.  If we define the linear map $\phi:A^2_q\longrightarrow A^2_q$ by 
$$\phi(x)= ax,\; \phi(y)= ay,\; \phi(x^{-1})= bx^{-1},\; \phi(y^{-1})= by^{-1},\quad a,b\in k,$$
then we have  the Hom-$\rho$-commutative algebra $A^2_q$. Note that $(A^2_q,.,\rho.\phi)$ is a Hom-associative $\rho$-commutative algebra if $a=b$. The set $\rho\text{-}DerA^q_2$ of all $\phi$-$\rho$-derivations on $A^q_2$ is a 
$A^2_q$-bimodule generated by 
$\frac{\partial}{\partial x}$ and $\frac{\partial}{\partial y}$, and 
$\Omega^1(A)$ is generated by $dx$, $dy$ such that 
$d x_j(\frac{\partial}{\partial x_i}) =\frac{\partial}{\partial x_i}(x_j) =\delta_{i,j},$ 
$\vert \frac{\partial}{\partial x_i}\vert =-\vert x_i\vert$ and $\vert d x_i\vert =\vert x_i\vert$,  where $x_1=x$
and $x_2 =y$.  We define the linear map
$\phi_{A^2_q}:\rho\text{-}DerA^2_q\longrightarrow\rho\text{-}DerA^2_q$ by 
$$\phi_{A^2_q}(\frac{\partial}{\partial x}) =\lambda\frac{\partial}{\partial x}+\gamma\frac{\partial}{\partial y},\quad \phi_{A^2_q}(\frac{\partial}{\partial y})=(\lambda+\gamma)\frac{\partial}{\partial y}.$$
where $\lambda, \gamma \in k$. All homogeneous metrics on $\rho\text{-}DerA^2_q$ were defined by \cite{N}	$$g=g_{11}dx\otimes_{\rho} dx +g_{12}dx\otimes_{\rho} dy+qg_{21} dy\otimes_{\rho} dx+g_{22}dy\otimes_{\rho} dy.$$
Now, by using metric $g$, we try to find the following relation
$$g(X,Y) =g(\phi_{A^2_q}(X),\phi_{A^2_q}(Y)),\quad \forall X,Y\in Hg(\rho\text{-}DerA^2_q).$$
	
{\bf Case 1:} If $X=\frac{\partial}{\partial x}$ and $Y=\frac{\partial}{\partial x}$, we have $g(\frac{\partial}{\partial x},\frac{\partial}{\partial x})=g_{11}$. Also, we get
\begin{align*}
g(\phi_{A^2_q}(\dfrac{\partial}{\partial x}) ,\phi_{A^2_q}(\dfrac{\partial}{\partial x}))&=g(\lambda\dfrac{\partial}{\partial x}+\gamma\dfrac{\partial}{\partial y},\lambda\dfrac{\partial}{\partial x}+\gamma\dfrac{\partial}{\partial y})\\	&=\lambda^2g(\dfrac{\partial}{\partial x},\dfrac{\partial}{\partial x}) +\lambda\gamma g(\dfrac{\partial}{\partial x},\dfrac{\partial}{\partial y}) +\gamma\lambda g(\dfrac{\partial}{\partial y},\dfrac{\partial}{\partial x})\\
&\ \ \  +\gamma^2g(\dfrac{\partial}{\partial y},\dfrac{\partial}{\partial y}) =\lambda^2g_{11}+\lambda\gamma(1+q)g_{12} +\gamma^2g_{22}.
\end{align*}
In this case, we find $\lambda^2=1$, $\lambda\gamma=0$, $\gamma^2=0$, so $\lambda=\pm 1$, $\gamma=0$.\\

{\bf Case 2:} If $X=\frac{\partial}{\partial y}$ and $Y=\frac{\partial}{\partial y}$, we have
\[
\begin{cases}
g(\phi_{A^2_q}(\dfrac{\partial}{\partial y}) ,\phi_{A^2_q}(\dfrac{\partial}{\partial y}))&=g((\lambda +\gamma)\dfrac{\partial}{\partial y},(\lambda +\gamma)\dfrac{\partial}{\partial y}) =(\lambda+\gamma)^2 g(\dfrac{\partial}{\partial y} ,\dfrac{\partial}{\partial y}) =(\lambda+\gamma)^2 g_{22},\\
g(\frac{\partial}{\partial y}, \frac{\partial}{\partial y})=g_{22}.&
\end{cases}
\]
In this case, we find $(\lambda+\gamma)^2=1$.\\
	
{\bf Case 3:} If $X=\frac{\partial}{\partial x}$ and $Y=\frac{\partial}{\partial y}$, we get $g(\frac{\partial}{\partial x}, \frac{\partial}{\partial y})=g_{21}=qg_{12}$ and 
\begin{align*}
g(\phi_{A^2_q}(\dfrac{\partial}{\partial x}) ,\phi_{A^2_q}(\dfrac{\partial}{\partial y}))&=g(\lambda\dfrac{\partial}{\partial x}+\gamma\dfrac{\partial}{\partial y},(\lambda+\gamma)\dfrac{\partial}{\partial y}) =\lambda(\lambda+\gamma) g(\dfrac{\partial}{\partial x} ,\dfrac{\partial}{\partial y})\\
&\ \ \ +\gamma(\lambda+\gamma) g(\dfrac{\partial}{\partial y} ,\dfrac{\partial}{\partial y})=q\lambda(\lambda+\gamma) g_{12}+\gamma(\lambda+\gamma)g_{22}.
\end{align*}
In this case, we find  $\lambda(\lambda+\gamma)=1$, $\gamma(\lambda+\gamma)=0$.\\
	
{\bf Case 4:} If $X=\frac{\partial}{\partial y}$ and $Y=\frac{\partial}{\partial x}$, we have
\[
\begin{cases}
g(\phi_{A^2_q}(\dfrac{\partial}{\partial y}) ,\phi_{A^2_q}(\dfrac{\partial}{\partial x}))&=g((\lambda +\gamma)\dfrac{\partial}{\partial y},\lambda \dfrac{\partial}{\partial x}+\gamma\dfrac{\partial}{\partial y}) =\lambda(\lambda+\gamma) g(\dfrac{\partial}{\partial y} ,\dfrac{\partial}{\partial x}) \\
&\ \ \ +\gamma(\lambda+\gamma) g(\dfrac{\partial}{\partial y} ,\dfrac{\partial}{\partial y})=\lambda(\lambda+\gamma) g_{12}+\gamma(\lambda+\gamma)g_{22},\\
g(\frac{\partial}{\partial y},\frac{\partial}{\partial x})=g_{12}.&
\end{cases}
\]
In this case, we also find $\lambda(\lambda+\gamma)=1$, $\gamma(\lambda+\gamma)=0$.\\
	
By the above cases, If $\gamma=0$ and $\lambda=\pm 1$, we obtain
\[\phi_{A^2_q}(\frac{\partial}{\partial x})=\pm\frac{\partial}{\partial x}, \quad \phi_{A^2_q}(\frac{\partial}{\partial y})=\pm\frac{\partial}{\partial y}.\]

Therefore $\phi_{A^2_q}$ can be write in the  following two cases

\[
\begin{cases}
\bullet\ \ 
 \phi_{A^2_q}(\dfrac{\partial}{\partial x})=-\dfrac{\partial}{\partial x},\ \ \phi_{A^2_q}(\dfrac{\partial}{\partial y})=-\dfrac{\partial}{\partial y}, \\ 
\bullet\ \ 
 \phi_{A^2_q}(\dfrac{\partial}{\partial x})=\dfrac{\partial}{\partial x},\ \ \phi_{A^2_q}(\dfrac{\partial}{\partial y})=\dfrac{\partial}{\partial y}.&
 \end{cases}
\]
\\
Also, we know that $\rho\text{-}DerA^q_2$ with $\rho$-commutator 
$[X,Y] =XY-\rho(X,Y)YX$ and the linear map $\phi_{A^2_q} $ is a Hom-$\rho$-commutative Lie algebra. But,
note that for the elements of the basis of  
$\rho\text{-}DerA^q_2$,
we have
\begin{equation}\label{10}
[\frac{\partial}{\partial x},\frac{\partial}{\partial y}]_{\rho}=\frac{\partial}{\partial x}\frac{\partial}{\partial y}-\rho(\frac{\partial}{\partial x},\frac{\partial}{\partial y})\frac{\partial}{\partial y}\frac{\partial}{\partial x}.
\end{equation}
On the other hand, since 
$[\frac{\partial}{\partial x},\frac{\partial}{\partial y}]_{\rho}\in\rho\text{-}DerA^2_q$, then we can write 
$$[\frac{\partial}{\partial x},\frac{\partial}{\partial y}]_{\rho}=p\dfrac{\partial}{\partial x}+q\dfrac{\partial}{\partial y},$$
where $p,q\in A$, so
$[\frac{\partial}{\partial x},\frac{\partial}{\partial y}]_{\rho}(x)=p$ and $[\frac{\partial}{\partial x},\frac{\partial}{\partial y}]_{\rho}(y)=q$.
But, \eqref{10} gives us
$[\frac{\partial}{\partial x},\frac{\partial}{\partial y}]_{\rho}(x)=[\frac{\partial}{\partial x},\frac{\partial}{\partial y}]_{\rho}(y)=0$.\\

Let us setting $\tilde{g}_{mk}:=\rho(\partial_m,\partial_k)g_{mk}=\rho(x_m,x_k)g_{mk}$ and continue with $g_{11}=x^{-2}, g_{12}=x^{-1}y^{-1}, g_{22} =y^{-2}$. In this case, $g$ is a homogeneous metric on $A^2_q$ (of degree $(0, 0)$) if and only if   
$$D=\dfrac{1-q^2}{q^2}x^{-2}y^{-2},$$
is invertible, in other words if and only if $q\neq 0,1, -1$.  
So, by the definition of $\tilde{g}$ we have
$$(\tilde{g}_{mk})=\left(\begin{array}{cc}
x^{-2} & qx^{-1}y^{-1}\\ x^{-1}y^{-1} & y^{-2}
\end{array}\right),$$
and
$$(\tilde{g}^{mk})=(\tilde{g}_{mk})^{-1}=\frac{1}{1-q^2}\left(\begin{array}{cc}
x^2 & -qxy\\ -xy & y^2
\end{array}\right).$$
Here, it seems that, we are ready to find the $\rho$-Christoffel coefficients $\Gamma_{ij}^t$ corresponding to each $\phi_{A^2_q}$. 
So, we need to consider the following  cases:\\
If
$\phi_{A^2_q}(\frac{\partial}{\partial x})=-\frac{\partial}{\partial x},\; \phi_{A^2_q}(\frac{\partial}{\partial y})=-\frac{\partial}{\partial y},$
we have
 $$\Gamma_{11}^1=x^{-1}, \Gamma_{22}^2=y^{-1},\Gamma_{12}^1=\Gamma_{21}^1=\Gamma_{12}^2=\Gamma_{21}^2=\Gamma_{11}^2=\Gamma_{22}^1=0.$$
 
 If $\phi_{A^2_q}(\frac{\partial}{\partial x})=\frac{\partial}{\partial x},\; \phi_{A^2_q}(\frac{\partial}{\partial y})=\frac{\partial}{\partial y},$
 we can find the following $\rho$-Christoffel coefficients
 $$\Gamma_{11}^1=-x^{-1}, \Gamma_{22}^2=-y^{-1},\Gamma_{12}^1=\Gamma_{21}^1=\Gamma_{12}^2=\Gamma_{21}^2=\Gamma_{11}^2=\Gamma_{22}^1=0.$$
\end{example}
\begin{definition}
Let $(A, \cdot,\rho, \phi)$ be a multiplicative Hom-$\rho$-commutative algebra and $(\rho\text{-}DerA,[\cdot,\cdot]_{\rho},\rho,\phi_A)$ be a multiplicative Hom-$\rho$-Lie algebra, where $\rho\text{-}DerA$ is the set of all $\rho$-derivation on Hom-$\rho$-commutative algebra $A$. For a Levi-Civita connection $\nabla$, we define the covariant derivation of $R$ as follows
\begin{align}\label{key}
(\nabla_ZR)(X, Y ) &= \nabla_{\phi_A^2(Z)}R(X, Y )(\cdot) - R(\nabla_ZX, \phi_A(Y) )\phi_A(\cdot)\\
& - \rho(Z, X)R(\phi_A(X),\nabla_ZY )\phi_A(\cdot)-\rho(Z, X + Y )R(\phi_A(X), \phi_A(Y))\nabla_Z(\cdot).\nonumber
\end{align}
\end{definition}
\begin{lemma}
For $X,Y,Z,V,W\in\rho\text{-}DerA$, the curvature $R$ satisfies the following equalities
\begin{enumerate}
\item[(a)]
$\rho(X,Y)R_{YZ}X + \rho(Y,Z)R_{ZX}Y + \rho(Z,X)R_{XY}Z = 0\;(Bianchi 1),$
\item[(b)]
\text{Second Bianchi identity for a torsion-free connection}\\
$\rho(V,X)R(X, Y, V,W) + \rho(Y,V)R(V,X, Y,W) + \rho(X,Y)R(Y, V,X,W) = 0,$
\item[(c)]
If $\nabla_{\phi_A(X)}\phi_A(Y)=\phi_A(\nabla_XY)$, then\\
$\rho(Y,Z)(\nabla_ZR)(X,Y) + \rho(X,Y)(\nabla_Y R)(Z,X) +\rho(Z,X)(\nabla_XR)(Y,Z) = 0$,
\end{enumerate}
where $R(X, Y, V,W) := g(R_{XY}V,W)$.
\end{lemma}
\begin{proof}
a) By the definition of curvature $R$, we can find the following relations
\begin{align}\label{nnn}
 \rho(Y,Z)R_{ZX}Y&=\rho(Y,Z)\{\nabla_{\phi_A(Z)}\nabla_XY-\rho(Z,X)\nabla_{\phi_A(X)}\nabla_ZY-\nabla_{[Z,X]_{\rho}}\phi_A(Y)\}\nonumber\\
&= \rho(Y,Z)\nabla_{\phi_A(Z)}\nabla_XY-\rho(Y,Z)\rho(Z,X)\rho(Z,Y)\nabla_{\phi_A(X)}\nabla_YZ\nonumber\\
&\ \ \ -\rho(Y,Z)\rho(Z,X)\nabla_{\phi_A(X)}[Z,Y]_{\rho}-\rho(Y,Z)\rho(Z,Y)\rho(X,Y)\nabla_{\phi(Y)}[Z,X]_{\rho}\nonumber\\
&\ \ \ -\rho(Y,Z)[{[Z,X]_{\rho}},\phi_A(Y)]_{\rho},
\end{align}
\begin{align}\label{mmm}
 \rho(X,Y)R_{YZ}X&=\rho(X,Y)\{\nabla_{\phi_A(Y)}\nabla_ZX-\rho(Y,Z)\nabla_{\phi_A(Z)}\nabla_YX-\nabla_{[Y,Z]_{\rho}}\phi_A(X)\}\nonumber\\
&= \rho(X,Y)\nabla_{\phi_A(Y)}\nabla_ZX-\rho(X,Y)\rho(Y,Z)\rho(Y,X)\nabla_{\phi_A(Z)}\nabla_XY\nonumber\\
&\ \ \ -\rho(X,Y)\rho(Y,Z)\nabla_{\phi_A(Z)}[Y,X]_{\rho}-\rho(X,Y)\rho(Y,X)\rho(Z,X)\nabla_{\phi_A(X)}[Y,Z]_{\rho}\nonumber\\
&\ \ \ -\rho(X,Y)[{[Y,Z]_{\rho}},\phi_A(X)]_{\rho},
\end{align}
and
\begin{align}\label{bbb}
 \rho(Z,X)R_{XY}Z&=\rho(Z,X)\{\nabla_{\phi_A(X)}\nabla_YZ-\rho(X,Y)\nabla_{\phi_A(Y)}\nabla_XZ-\nabla_{[X,Y]_{\rho}}\phi_A(Z)\}\nonumber\\
&= \rho(Z,X)\nabla_{\phi_A(X)}\nabla_YZ-\rho(Z,X)\rho(X,Y)\rho(X,Z)\nabla_{\phi_A(Y)}\nabla_ZX\nonumber\\
&\ \ \ -\rho(Z,X)\rho(X,Y)\nabla_{\phi_A(Y)}[X,Z]_{\rho}-\rho(Z,X)\rho(X,Z)\rho(Y,Z)\nabla_{\phi_A(Z)}[X,Y]_{\rho}\nonumber\\
&\ \ \ -\rho(Z,X)[{[X,Y]_{\rho}},\phi_A(Z)]_{\rho}.
\end{align}
Summing the relations \eqref{nnn}, \eqref{mmm} and \eqref{bbb}, imply
$$\rho(X,Y)R_{YZ}X + \rho(Y,Z)R_{ZX}Y + \rho(Z,X)R_{XY}Z = 0.$$
The second Bianchi identity follows immediately from the first and 
the relation (c) follows from a direct calculation by applying the relation $\nabla_{\phi_A(X)}\phi_A(Y)=\phi_A(\nabla_XY)$.
\end{proof}	

Based on our knowledge of Riemannian geometry the following properties hold for Riemannian curvature tensor $R(X,Y,Z,W)$
\begin{enumerate}
\item[(1)] 
$R$ is skew-symmetric in the first two and last two entries
\begin{center}
$\ \ \ R(X, Y,Z,W) = -R(Y,X,Z,W) = -R(X,Y,W,Z)$.
\end{center}
\item[(2)] $R$ is symmetric between the first two and last two entries
\begin{center}
$R(X, Y,Z,W) = R(Z,W,X, Y )$.
\end{center}
\end{enumerate}	
Also, in \cite{N}, Ngakeu presented that the curvature tensor $R(X,Y,Z,W)$ on $\rho$-commutative algebras has the same properties of Riemannian geometry, that is
\begin{enumerate}
\item[(i)]
$R(X,Y,V,W) = -\rho(X,Y)R(Y,X,V,W) = -\rho(V,W)R(X,Y,W,V),$
\item[(ii)]
$R(X,Y,V,W) = \rho(X+Y,V+W)R(V,W,X,Y).$
\end{enumerate}
For the curvature tensor $R(X,Y,V,W)$ on Hom-$\rho$-commutative algebras necessarily, we can not have the following properties
\[R(X,Y,V,W)=\rho(X+Y, V+W)R(V,W,X,Y),\]
\[R(X,Y,V,W) = -\rho(X,Y)R(Y,X,V,W) = -\rho(V,W)R(X,Y,W,V).\]
\section{Differential calculus, Poisson bracket}
In this section, we recall the notions of representation, cochains and Hom-cochains on Hom-$\rho$-Lie algebras. These notions are defined analogously of the classical case by K. Abdaoui, F. Ammar and A. Makhlouf in \cite{AAM}. Then, we try to develop differential calculus by using them and recall the notion of Poisson bracket on Hom-$\rho$-Lie algebras and investigate some examples (for the classical case see \cite{Y1}, \cite{N11}).

\begin{definition}\cite{AAM}
Let $(A,[\cdot,\cdot]_{\rho},\rho,\phi)$ be a Hom-$\rho$-Lie algebra. For any non-negative integer $k$, a $\phi^k$-$\rho$-derivation of degree $|X|$ on $A$ is a linear map $X : A \longrightarrow A$ 
such that\\
$$X\circ\phi=\phi\circ X,\quad i.e.,\quad [X,\phi]_{\rho}=0,$$\\
and
\begin{equation}
X[f,g]_{\rho}= [X(f),\phi^k(g)]_{\rho} + \rho(X,f)[\phi^k(f),X(g)]_{\rho}.
\end{equation}
\end{definition}
We denote by $\rho\text{-}Der_{\phi^k} A$ the space of all $\phi^k$-$\rho$-derivations of $A$. It is easy to see that for $X\in\rho\text{-}Der_{\phi^k} A$ and $Y\in\rho\text{-}Der_{\phi^s} A$
the $\rho$-commutator of $X,Y$, defined by $[X, Y ]_{\rho} = X\circ Y - \rho(X,Y)Y\circ X$, is
a $\phi^{k+s}$-$\rho$-derivation and 
$$\rho\text{-}Der_{\phi}A=\bigoplus_{k\geq 0} \rho\text{-}Der_{\phi^k}A,$$
is a $\rho$-Lie algebra with above bracket.
\begin{proposition}
Let $(A,[\cdot,\cdot]_{\rho},\rho,\phi)$ be a Hom-$\rho$-Lie algebra. The quadruple $(\rho$-$Der_{\phi}A,[\cdot,\cdot]_{\rho},\rho,\phi_A)$ consisting of the space $\rho$-$Der_{\phi}A$, two-cycle $\rho$, $\rho$-commutator $[X,Y]_{\rho}=X\circ Y-\rho(X,Y)Y\circ X$ and even linear map $\phi_A:\rho$-$Der_{\phi}A\longrightarrow
\rho$-$Der_{\phi}A$ given by $\phi_A(X)=X\circ \phi$
is a Hom-$\rho$-Lie algebra.
\end{proposition}
\begin{proof}
We know that the composition of functions is always associative, that is if $h_1$, $h_2$ and $h_3$ are three functions with suitably chosen domains and codomains, then $h_1\circ (h_2\circ h_3)=(h_1\circ h_2)\circ h_3$. Let us use this to prove our proposition. So, we have 
\begin{align*}
\phi_A(Y)\circ(Z\circ X)&=(\phi_A(Y)\circ Z)\circ X=((Y\circ \phi)\circ Z)\circ X=(Y\circ(\phi\circ Z))\circ X\\
&=Y\circ((\phi\circ Z)\circ X)=Y\circ((Z\circ\phi)\circ X)=Y\circ(Z\circ(\phi\circ X))\\
&=Y\circ(Z\circ(X\circ\phi))=Y\circ(Z\circ\phi_A(X))=(Y\circ Z)\circ\phi_A(X).
\end{align*}
Now, by using the relation $\phi_A(Y)\circ(Z\circ X)=(Y\circ Z)\circ\phi_A(X)$ for $X,Y,Z\in Hg(\rho\text{-}Der_{\phi}A)$, we study the Jacobi-identity. Direct calculations give us
\begin{align*}
\rho(Z,X)[\phi_A(X),[Y,Z]_{\rho}]_{\rho}&=\rho(Z,X)\phi_A(X)\circ(Y\circ Z)-\rho(Z,X)\rho(X,Y+Z)(Y\circ Z)\circ\phi_A(X)\\
&\ \ \ -\rho(Z,X)\rho(Y,Z)\phi_A(X)\circ(Z\circ Y)\\
&\ \ \ +\rho(Z,X)\rho(X,Y+Z)\rho(Y,Z)(Z\circ Y)\circ \phi_A(X),
\end{align*}
\begin{align*}
\rho(X,Y)[\phi_A(Y),[Z,X]_{\rho}]_{\rho}&=\rho(X,Y)\phi_A(Y)\circ(Z\circ X)-\rho(X,Y)\rho(Y,X+Z)(Z\circ X)\circ\phi_A(Y)\\
&\ \ \ -\rho(X,Y)\rho(Z,X)\phi_A(Y)\circ(X\circ Z)\\
&\ \ \ +\rho(X,Y)\rho(Y,X+Z)\rho(Z,X)(X\circ Z)\circ \phi_A(Y),
\end{align*}
\begin{align*}
\rho(Y,Z)[\phi_A(Z),[X,Y]_{\rho}]_{\rho}&=\rho(Y,Z)\phi_A(Z)\circ(X\circ Y)-\rho(Y,Z)\rho(Z,Y+X)(X\circ Y)\circ\phi_A(Z)\\
&\ \ \ -\rho(Y,Z)\rho(X,Y)\phi_A(Z)\circ(Y\circ X)\\
&\ \ \ +\rho(Y,Z)\rho(Z,Y+X)\rho(X,Y)(Y\circ X)\circ \phi_A(Z).
\end{align*}
In the end, summing three above equations imply the Jacobi identity.
\end{proof}
\begin{definition}\cite{AAM}
Let $V$ be a vector space. A linear map $\mu:A\longrightarrow {\rm End}(V)$ is called a representation of the Hom-$\rho$-Lie algebra $(A,[\cdot,\cdot]_{\rho},\rho,\phi)$	on $V$ with respect to $B\in {\rm End}(V)$ if the following equality is satisfied
\[
\mu[f,g]_{\rho}\circ B=\mu(\phi(f))\circ\mu(g)-\rho(f,g)\mu(\phi(g))\circ\mu(f).
\]
Moreover, a representation $(V, \mu)$ is said to be graded if $V = \bigoplus_{a\in G}V_a$ is a $G$-graded space such that
$$\mu(f)(V_a)\subseteq V_{|f|+a},$$
for all the homogeneous elements $f\in A$ and $a\in G$.
\end{definition}
Let $V$ be a $G$-graded vector space and $B:V\longrightarrow V$ be an even homomorphism.
\begin{definition}\cite{Y1}
A $k$-cochain on a Hom-$\rho$-Lie algebra $(A,[\cdot,\cdot]_{\rho},\rho,\phi)$ is a $\rho$-skew-symmetric and $k$-linear map $\alpha:A\times\cdots\times A\longrightarrow V$ of $G$-degree $|\alpha|$, in the sense of
\begin{align*}
\alpha(f_1,\cdots,f_k)\subset V_{|f_1|+\cdots+|f_k|+|\alpha|},
\end{align*}
where $f_1,\cdots,f_k\in Hg(A)$. We denote by $C^k(A;V)$ the set of $k$-cochains on $A$.

$\alpha\in C^k(A;V)$ is called a $k$-Hom-cochain on $A$ if for $f_1,\cdots,f_k\in Hg(A)$, the following relation holds
$$B(\alpha(f_1,\cdots,f_k))=\alpha(\phi(f_1),\cdots,\phi(f_k)).$$
Let $C^k_\phi(A,V)$ denotes the set of $k$-Hom-cochains on $A$. Then $C^k_\phi(A,V)$ is a graded algebra with $C^0_{\phi}(A,V)=V$ and we have
$$C_{\phi}(A,V)=\bigoplus_{k\geq 0}C^k_{\phi}(A,V).$$
\end{definition}
In the next, let $(A, [\cdot,\cdot]_{\rho},\rho, \phi)$ be a Hom-$\rho$-Lie algebra and $(\rho\text{-}Der_{\phi}A,[\cdot,\cdot]_{\rho},\rho,\phi_A)$ be Hom-$\rho$-Lie algebra of all $\phi^k$-$\rho$-derivations on Hom-$\rho$-Lie algebra $A$, where $\rho\text{-}Der_{\phi}A$ is equipped with the representation $\mu_A$ on $A$ $(\mu_A:\rho\text{-}Der_{\phi}A\longrightarrow {\rm End}(A))$ with respect to $B=Id_A:A\longrightarrow A$.\\
We intend to define some operators on the set of $k$-Hom-cochains  
$$C^k_{\phi_A}(\rho\text{-}Der_{\phi}A, A)=\{\alpha\in C^k(\rho\text{-}Der_{\phi}A, A):\quad \alpha\circ\phi_A=\alpha\},$$
on $\rho\text{-}Der_{\phi}A$.

Now, we define the co-boundary operator $d_A:C^k_{\phi_A}(\rho\text{-}Der_{\phi}A, A)\longrightarrow C^{k+1}_{\phi_A}(\rho\text{-}Der_{\phi}A, A)$ by
$$d_Af(X)=\mu_A(\phi_A^{-1}(X))\cdot f,\quad f\in A,$$
and 
\begin{align}\label{8}
d_A\alpha(X_1,\cdots,X_{k+1})=:&\sum^{k+1}_{j=1}(-1)^{j-1}\rho(\sum^{j-1}_{i=1}X_i ,X_j)\mu_A(\phi_A^{k-1}(X_j)) \cdot\alpha(X_1, \cdots , \widehat{X_j} , \cdots ,X_{k+1})\\
&\ \ \ +\sum_{1\leqslant j<l \leqslant k+1}(-1)^{j+l}\rho(\sum^{j-1}_{i=1}X_i , X_j)\rho(\sum^{j-1}_{i=1} X_i , X_l)\nonumber\\
&\times\rho(\sum^{l-1}_{i=j+1}X_i , X_l)\alpha([X_j ,X_l]_{\rho}, \phi_A(X_1),\cdots, \widehat{\phi_A(X_j)} ,\cdots, \widehat{\phi_A(X_l)},\cdots ,\phi_A(X_{k+1})),\nonumber
\end{align}
for $k\geq 1$, $\alpha\in C^k_{\phi_A}(\rho\text{-}Der_{\phi}A, A)$ and $X_l\in Hg(\rho\text{-}Der_{\phi}A),~~~ l\in \{1,\cdots,k+1\}$, where 
$\widehat{X_j}$ means that $X_j$ is omitted. Note that $|d_A\alpha|=|\alpha|$ and $d^2_A=0$ (the condition $d_A^2=0$ does not follow if the condition $\alpha\circ\phi_A=\alpha$ is omitted, so it is necessary to define the differential operators on $k$-Hom-cochains).\\

The inner and Lie derivations also are defined on $C_{\phi_A}(\rho\text{-}Der_{\phi}A, A)$ by
\begin{align*}
i_X \alpha(X_1,\cdots ,X_{k-1}) :& = \rho(\sum^{k-1}_{i=1}X_i , X)\alpha(X,X_1,\cdots ,X_{k-1}),\quad i_X(f)=0,\\
L_X=i_X\circ d_A+d_A\circ i_X,
\end{align*}
where $X_l\in Hg(\rho\text{-}Der_{\phi}A),~~~ l\in \{1,\cdots,k\}$. 
Note that $ |i_X| = |L_X| = |X|$.
\begin{remark}
For the inner and Lie derivations $i_X$ and $L_X$, we have $i_X\circ i_Y +\rho(X,Y) i_Y\circ i_X=0$ and $d\circ L_X=L_X\circ d$, where $X,Y\in Hg(\rho\text{-}Der_{\phi}A)$, But the property $[L_X, i_Y]=i_{[X,Y]}$ not necessarily holds. For instance, if we consider the extended-hyperplane $A^2_q$ introduced in Example \ref{12}, if we set $a=b=1$ then $(A^2_q,\cdot,\rho,\phi=Id, [\cdot, \cdot]_{\rho}=0)$ is a Hom-$\rho$-Lie algebra. For $\phi=Id$ and $k=0$, $\frac{\partial}{\partial x}$ and $\frac{\partial}{\partial y}$ are $\phi^0$-$\rho$-derivation. Let us set $\mu_{A^2_q}=ad$ and consider $\phi_{A^2_q}(\frac{\partial}{\partial x})=-\frac{\partial}{\partial x}$ and $\phi_{A^2_q}(\frac{\partial}{\partial y})=-\frac{\partial}{\partial y}$. Thus for $X=\frac{\partial}{\partial x}$, $Y=\frac{\partial}{\partial y}$, $Z=\frac{\partial}{\partial x}$ we have
\begin{flalign*}
[i_X\circ d+d\circ i_X, i_Y](\alpha)(Z)&=([i_X\circ d, i_Y]+[d\circ i_X, i_Y])(\alpha)(Z)&\\
&=(1-q)[\frac{\partial}{\partial x}, \alpha(\frac{\partial}{\partial y},\frac{\partial}{\partial x})]+2[\frac{\partial}{\partial x}, \alpha(\frac{\partial}{\partial x},\frac{\partial}{\partial y})]&\\
&\neq i_{[\frac{\partial}{\partial x}, \frac{\partial}{\partial y}]}=i_{[X,Y]}=0.
\end{flalign*}
\end{remark}
\begin{definition}
We say that Hom-$\rho$-Lie algebra $(A,[\cdot,\cdot]_{\rho},\rho,\phi)$ satisfies in the Cartan identity if the following condition holds
$$[L_X, i_Y]=i_{[X,Y]}.$$
\end{definition}
In this case, one can easily show that
$$L_{[X,Y]_{\rho}}=[L_X,L_Y]_{\rho}.$$
\begin{definition}
Let  $\Omega\in C^2_{\phi_A}(\rho\text{-}Der_{\phi}A,A)$ 
is called a sympelectic structure on 
$A$
if  $\Omega$ is non-degenerate and closed ($d_A\Omega=0$).
\end{definition}
Let us define 
$\tilde{\Omega} :\rho\text{-}Der_{\phi}A\longrightarrow C^1_{\phi_A}(\rho\text{-}Der_{\phi}A,A)$ by  $\tilde{\Omega}(X)=\Omega(.,\phi_A(X)).$ Then the nondegenerate property of $\Omega$ is equivalent to the assertion that $\tilde{\Omega}$ is isomorphism (note that, the definition  of non degeneracy creates a one-to-one correspondence between $\rho\text{-}Der_{\phi}A$ and $C^1_{\phi_A}(\rho\text{-}Der_{\phi}A,A)$ and they have tha same cardinality. Thus if they are even infinite dimensional, then this isometric is meaningful).
It is remarkable that $\Omega$
is homogeneous iff $\tilde{\Omega}$
is homogeneous and we have $\vert \Omega\vert =\vert\tilde{\Omega}\vert.$
\begin{definition}
 $X\in\rho\text{-}Der_{\phi}A$ is called a locally Hamiltonian $\phi^k$-$\rho$-derivation if $L_{\phi_A(X)}\Omega=0$.
\end{definition}
\begin{lemma}
	Let $(\rho\text{-}Der_{\phi}A,[\cdot,\cdot]_{\rho},\rho,\phi_A)$ satisfies the Catran identity.
$X\in\rho\text{-}Der_{\phi}A$ is locally Hamiltonian if and only if $d_A(i_{\phi_A(X)}\Omega)=0$.
\end{lemma}
\begin{proof}
By $L_{\phi_A(X)}=d_A\circ i_{\phi_A(X)}+
i_{\phi_A(X)}\circ d_A$ the proof is clear.
\end{proof}
By the relation  $L_{\phi_A[X,Y]_{\rho}}=[L_{\phi_A(X)},L_{\phi_A(Y)}]_{\rho}$, we can show that if $X,Y\in\rho\text{-}Der_{\phi}A$ are locally Hamiltonian, then $[X,Y]_{\rho}$ is also locally Hamiltonian.
\begin{definition}
For any $f \in A$, the vector
 $X :=\tilde{\Omega}^{-1}(d_Af)$ is called the Hamiltonian $\phi^k$-$\rho$-derivation associated to $f$.
 \end{definition}
 Let us $X_f$ denotes the Hamiltonian $\phi^k$-$\rho$-derivation associated to $f$, i.e., $X=X_f$. So $X_f$ is of $G$-degree $|X_f|=|f|-|\Omega|$
and
$$ d_Af =\Omega(., \phi_A(X_f))= -i_{\phi_A(X_f)}\Omega.$$
In other words
 $$\Omega(\phi_A(X_g) ,\phi_A(X_f)) = \mu_A(X_g)\cdot f.$$
 Note that, since $\Omega$ is a $2$-Hom-cochain, then 
 $$\Omega(\phi_A(X) ,\phi_A(Y))=\Omega(X,Y),$$
 and so
 $$\Omega(X_g,X_f)=\mu_A(X_g)\cdot f.$$
Let the following two relations exist between
the maps $\phi$ and $\phi_A$ and $\mu_A(X)\in End(A)$
\begin{align}
\mu_A(\phi_A(X_f))&=\mu_A(X_{\phi(f)}),\label{15}\\
\mu_A(\phi_A(X_f))\cdot[g,h]_{\rho}&=[\mu_A(X_f)\cdot g,\phi(h)]_{\rho}+\rho(X_f,g)[\phi(g),\mu_A(X_f)\cdot h]_{\rho}.\label{16}
\end{align}
\begin{lemma}\label{14}
Let $(\rho\text{-}Der_{\phi}A,[\cdot,\cdot]_{\rho},\rho,\phi_A)$ satisfies the Catran identity. If $X,Y\in \rho\text{-}Der_{\phi}A$ are locally Hamiltonians $\phi^{k}$-$\rho$-derivation and $\phi^{s}$-$\rho$-derivation respectively, then the following relation holds
$$[X,Y]_{\rho}=X_{\Omega(X,Y)}=-\rho(X,Y)X_{\Omega(Y,X)},$$
i.e., $[X,Y]_{\rho}$ is the Hamiltonian $\phi^{k+s}$-$\rho$-derivation associated to $\Omega(X,Y)$.
\end{lemma}
\begin{proof}
By the equality $i_{[X,Y]_{\rho}}=[L_X,i_Y]_{\rho}$ and equivalently  $i_{\phi_A[X,Y]_{\rho}}=[L_{\phi_A(X)},i_{\phi_A(Y)}]_{\rho}$, we have
\begin{align*}
i_{\phi_A[X,Y]_{\rho}}\Omega=[L_{\phi_A(X)},i_{\phi_A(Y)}]_{\rho}\Omega&=L_{\phi_A(X)}(i_{\phi_A(Y)}\Omega) -\rho(X,Y)i_{\phi_A(Y)}(L_{\phi_A(X)}\Omega).
\end{align*}
Since $X$ is a locally Hamiltonian $\phi^k$-$\rho$-derivation, we get
\begin{align*}
i_{\phi_A[X,Y]_{\rho}}\Omega=[L_{\phi_A(X)},i_{\phi_A(Y)}]_{\rho}\Omega&=L_{\phi_A(X)}(i_{\phi_A(Y)}\Omega).
\end{align*}
In the next, by the Cartan identity and given that the $\phi^s$-$\rho$-derivation $Y$ is locally Hamiltonian, we easily obtain the following relation
\begin{align*}
i_{\phi_A[X,Y]_{\rho}}\Omega=\rho(X,Y)d_A(\Omega(\phi_A(Y),\phi_A(X))=-d_A(\Omega(X,Y)),
\end{align*}
and so
\[
-i_{\phi_A[X,Y]_{\rho}}\Omega=d_A(\Omega(X,Y)).
\]
Thus, the conclusion is held, that is
\[
[X,Y]_{\rho}=X_{\Omega(X,Y)}=-\rho(X,Y)X_{\Omega(Y,X)}.
\]
\end{proof}
\begin{definition}
\cite{IIB} A Poisson Hom-$\rho$-algebra consists of a $G$-graded vector space $A$, bilinear maps $\cdot: A \times A \longrightarrow A $ and $\lbrace \cdot,\cdot\rbrace_{\rho} : A \times A \longrightarrow A$ of $G$-degree $|\{\cdot,\cdot\}_{\rho}|=P$, an even linear map $\phi : A \longrightarrow A$ and a two-cycle 
$\rho:G\times G\longrightarrow k^{\star}$ such that\\
(1) $(A,\cdot,\rho,\phi)$ is a Hom-associative $\rho$-algebra.\\
(2) $(A,\lbrace ·,·\rbrace_{\rho},\rho,\phi)$ is a Hom-$\rho$-Lie algebra, i.e.,
\begin{enumerate}
\item[(i)]
$| \{f,g\}_{\rho} | =P+|f|+|g|,$
\item[(ii)]
$\{f,g\}_{\rho}=-\rho(f,g)\{g,f\}_{\rho},$
\item[(iii)]
$\rho(h,f)\{\phi(f),\{g,h\}_{\rho}\}_{\rho}+\rho(g,h)\{\phi(h),\{f,g\}_{\rho}\}_{\rho}+\rho(f,g)\{\phi(g),\{h,f\}_{\rho}\}_{\rho}=0.$
\end{enumerate}
(3) For all $f,g,h\in A$, $ \lbrace f\cdot g,\phi(h)\rbrace_{\rho} = \rho(g,h+P)\lbrace f,h\rbrace_{\rho}\cdot\phi(g)+ \phi(f)\cdot\lbrace
g,h\rbrace_{\rho}.$\\
Furthermore, if $fg=\rho(f,g) gf$ for all $f,g\in Hg(A)$, then we have a Poisson Hom-$\rho$-commutative algebra.
\end{definition}
Equivalently, Poisson Hom-$\rho$-Lie algebra can be defined in the following expression:
\begin{definition}
A Poisson Hom-$\rho$-Lie algebra is a multiplex $(A,[\cdot,\cdot]_{\rho}, \lbrace\cdot,\cdot\rbrace_{\rho},\rho, \phi)$ consisting of a $G$-graded vector space $A$,
bilinear maps $[\cdot,\cdot]_{\rho} : A \times A \longrightarrow A $ and $\lbrace \cdot,\cdot\rbrace_{\rho} : A \times A \longrightarrow A$ of $G$-degree $|\{\cdot,\cdot\}_{\rho}|=P$, an even linear map $\phi : A \longrightarrow A$ and a two-cycle $\rho:G\times G\longrightarrow k^{\star}$ satisfying\\
(1) $(A, [\cdot,\cdot]_{\rho},\rho, \phi)$ is a Hom-$\rho$-Lie algebra.\\
(2) $(A, \lbrace\cdot,\cdot\rbrace_{\rho},\rho, \phi)$ is a Hom-$\rho$-Lie algebra.\\
(3) For all $f,g,h\in A$, $ \lbrace[f,g]_{\rho},\phi(h)\rbrace_{\rho} = \rho(g,h+P)[\lbrace f,h\rbrace_{\rho}, \phi(g)]_{\rho} + [\phi(f), \lbrace
g,h\rbrace_{\rho}]_{\rho}.$
\end{definition}
\begin{example}
Let $(A, \cdot, \lbrace\cdot,\cdot\rbrace_{\rho},\rho, \phi)$ be a Poisson Hom-$\rho$-algebra. We define the bracket $[\cdot,\cdot]:A\times A\longrightarrow A$ by $[f,g]=f\cdot g-\rho(f,g)g\cdot f$. In this case
$(A,[\cdot,\cdot]_{\rho}, \cdot, \lbrace\cdot,\cdot\rbrace_{\rho},\rho, \phi)$ is a Poisson Hom-$\rho$-Lie algebra.\\
It is enough to check that 
$ \lbrace[f,g]_{\rho},\phi(h)\rbrace_{\rho} = \rho(g,h+P)[\lbrace f,h\rbrace_{\rho}, \phi(g)]_{\rho} + [\phi(f), \lbrace
g,h\rbrace_{\rho}]_{\rho}.$ For this, we have
\begin{align*}
\lbrace[f,g]_{\rho},\phi(h)\rbrace_{\rho}&=\{f\cdot g, \phi(h)\}_{\rho}-\rho(f,g)\{g\cdot f,\phi(h)\}_{\rho},
\end{align*}
On the other hand, since $(A, \cdot, \lbrace\cdot,\cdot\rbrace_{\rho},\rho, \phi)$ is a Poisson Hom-$\rho$-algebra, then
\[
\lbrace f\cdot g,\phi(h)\rbrace_{\rho} = \rho(g,h+P)\lbrace f,h\rbrace_{\rho}\cdot\phi(g)+ \phi(f)\cdot\lbrace g,h\rbrace_{\rho},
\]
and 
\[
\lbrace g\cdot f,\phi(h)\rbrace_{\rho} = \rho(f,h+P)\lbrace g,h\rbrace_{\rho}\cdot\phi(f)+ \phi(g)\cdot\lbrace f,h\rbrace_{\rho}.
\]
Therefore
\begin{align*}
\lbrace[f,g]_{\rho},\phi(h)\rbrace_{\rho}&=\rho(g,h+P)\lbrace f,h\rbrace_{\rho}\cdot\phi(g)+ \phi(f)\cdot\lbrace g,h\rbrace_{\rho}\\
&\ \ \ -\rho(f,g)\{\rho(f,h+P)\lbrace g,h\rbrace_{\rho}\cdot\phi(f)+ \phi(g)\cdot\lbrace f,h\rbrace_{\rho}\}\\
&=\rho(g,h+P)[\{f,h\}_{\rho},\phi(g)]_{\rho}+[\phi(f),\{g,h\}_{\rho}]_{\rho}.
\end{align*}
\end{example}
\begin{theorem}
Let $(A,[\cdot,\cdot]_{\rho},\rho,\phi)$ be a Hom-$\rho$-Lie algebra and $\Omega$ be the homogeneous symplectic structure. Defining the $\rho$-Poisson
bracket $\lbrace \cdot,\cdot\rbrace_{\rho}$ associated to $\Omega$ as
\begin{equation*}
\lbrace f, g\rbrace_{\rho} :=-\rho(\Omega,g)\mu_A(X_f)\cdot g =-\rho(\Omega,g)\Omega( \phi_A(X_f),\phi_A(X_g))\quad f, g \in Hg(A),
\end{equation*}
$(A,[\cdot,\cdot]_{\rho},\{\cdot,\cdot\}_{\rho},\phi)$ is a Poisson Hom-$\rho$-Lie algebra.
\end{theorem}
\begin{proof}
Since $\Omega$ is a $2$-Hom-cochain, then we have
\begin{equation}\label{7}
\lbrace f, g\rbrace_{\rho} :=-\rho(\Omega,g)\mu_A(X_f)\cdot g =-\rho(\Omega,g)\Omega( X_f,X_g)\quad f, g \in Hg(A).
\end{equation}
At first, we show that $(A,\{\cdot,\cdot\}_{\rho},\rho,\phi)$ is a Hom-$\rho$-Lie algebra. We have
$$|\{f,g\}_{\rho}|=|\{\cdot,\cdot\}_{\rho}|+|f|+|g|.$$
On the other hand, we have
$$|\Omega( X_f,X_g)|=|\Omega|+|X_f|+|X_g|=P+|f|+|g|.$$
By the definition of $\rho$-Poisson bracket, since
$|\{f,g\}_{\rho}|=|\Omega( X_f,X_g)|$, then we can find
$|\{.,.\}_{\rho}|=P=-|\Omega|$.\\
Now, we investigate the following relation
$$\{f,g\}_{\rho}=-\rho(f,g)\{g,f\}_{\rho}.$$
By the relation \eqref{7}, we have
\begin{align*}
\{f,g\}_{\rho}&=-\rho(\Omega,g)\Omega(X_f,X_g)=\rho(\Omega,g)\rho(X_f,X_g)\Omega(X_g,X_f)\\
&=\rho(\Omega,g)\rho(f,g)\rho(f,-\Omega)\rho(-\Omega,g)\Omega(X_g,X_f)\\
&=\rho(f,g)\rho(\Omega,f)\Omega(X_g,X_f)=-\rho(f,g)\{g,f\}_{\rho}.
\end{align*} 
Now, it's time to complete the proof by showing that the $\rho$-Hom-jacobi identity holds. At first, note that by Lemma \ref{14}, we have
\begin{align}\label{9}
\Omega([X_f,X_g]_{\rho},\phi_A(X_h))&=\Omega(X_{\Omega(X_f,X_g)},\phi_A(X_h))=-\rho(X_f+X_g,X_h)\Omega(\phi_A(X_h),X_{\Omega(X_f,X_g)})\\
&=-\rho(X_f+X_g,X_h)\mu_A(\phi_A(X_h))\cdot\Omega(X_f,X_g)\nonumber\\
&=-\rho(X_f+X_g,X_h)\mu_A(\phi_A(X_h))\mu_A(X_f)\cdot g.\nonumber
\end{align}
Since $\Omega$ is a close form, then by \eqref{8}, we have
\begin{align*}
0=d_A\Omega(X_f,X_g,X_h)&=\mu_A(\phi_A(X_f))\Omega(X_g,X_h)-\rho(X_f,X_g)\mu_A(\phi_A(X_g))\Omega(X_f,X_h)\\
&\ \ \ +\rho(X_f+X_g,X_h)\mu_A(\phi_A(X_h))\Omega(X_f,X_g) -\Omega([X_f,X_g]_{\rho},\phi_A(X_h))\\
&\ \ \ +\rho(X_g,X_h)\Omega([X_f,X_h]_{\rho},\phi_A(X_g))-\rho(X_f,X_g+X_h)\Omega([X_g,X_h]_{\rho},\phi_A(X_f)).
\end{align*}
This case will complete by invoking \eqref{9} and again \eqref{7}, as
\begin{align*}
0=d_A\Omega(X_f,X_g,X_h)&=-2\rho(X_g,X_h)\mu_A(\phi_A(X_f))\mu_A(X_h)\cdot g\\
&\ \ \ +2\rho(X_f,X_g+X_h)\mu_A(\phi_A(X_g))\mu_A(X_h)\cdot f\\
&\ \ \ -2\rho(X_f+X_g,X_h)\rho(X_f,X_g)\mu_A(\phi_A(X_h))\mu_A(X_g)\cdot f.
\end{align*}
By \eqref{7} again, this time vice versa, we obtain
\begin{align*}
0=d_A\Omega(X_f,X_g,X_h)
&=\rho(g,\Omega)\rho(h,2\Omega)\rho(f,h)\rho(h,f)\{\phi(f),\{g,h\}_{\rho}\}_{\rho}\\
&\ \ \ \rho(f+g,h)\rho(g,\Omega)\rho(h,2\Omega)\{\phi(h),\{f,g\}_{\rho}\}_{\rho}\\
& \ \ \ \rho(f,g+h)\rho(g,\Omega)\rho(h,2\Omega)\{\phi(g),\{h,f\}_{\rho}\}_{\rho}.
\end{align*}
With this result, we can write
$$\rho(h,f)\{\phi(f),\{g,h\}_{\rho}\}_{\rho}+\rho(g,h)\{\phi(h),\{f,g\}_{\rho}\}_{\rho}+\rho(f,g)\{\phi(g),\{h,f\}_{\rho}\}_{\rho}=0.$$
We continue with the checking of the following relation
$$\lbrace[f,g]_{\rho},\phi(h)\rbrace_{\rho} = \rho(g,h+P)[\lbrace f,h\rbrace_{\rho}, \phi(g)]_{\rho} + [\phi(f), \lbrace
h,g\rbrace_{\rho}]_{\rho}.$$
We have
\begin{align*}
\{[f,g]_{\rho},\phi(h)\}_{\rho}&=-\rho(\Omega,h)\Omega(X_{[f,g]_{\rho}},X_{\phi(h)})\\
&=\rho(\Omega,h)\rho(f+g-\Omega,h-\Omega)\Omega(X_{\phi(h)},X_{[f,g]_{\rho}}).
\end{align*}
Now, by using relation \eqref{7}, we have
\begin{align*}
\{[f,g]_{\rho},\phi(h)\}_{\rho}
&=\rho(\Omega,h)\rho(f+g-\Omega,h-\Omega)\mu_A(X_{\phi(h)})\cdot [f,g]_{\rho}.
\end{align*}
In the next, the relations \eqref{15} and \eqref{16} give us
\begin{align*}
\{[f,g]_{\rho},\phi(h)\}_{\rho}
&=\rho(f+g,h)\rho(f+g,-\Omega)[\mu_A(X_h)\cdot f,\phi(g)]_{\rho}\\
&\ \ \ +\rho(f+g,h)\rho(f+g,-\Omega)\rho(X_h,f)[\phi(f),\mu_A(X_h)\cdot g]_{\rho}\\
&=\rho(g,h-\Omega)[\{f,h\}_{\rho},\phi(g)]_{\rho}+[\phi(f),\{g,h\}_{}]_{\rho}.
\end{align*}
\end{proof}
Lemma \ref{14} implies the following
\begin{corollary}
We have
\[
[X_f,X_g]_{\rho}=X_{\Omega(X_f,X_g)}=-\rho(g,\Omega)X_{\{f,g\}_{\rho}}.
\]
\end{corollary}
The above theorem is the key to constructing Hamiltonian derivation and Poisson bracket on the specific Hom-$\rho$-commutative algebras. To illustrate an application of this lemma, let us state the following example.\\

\begin{example} 
Let's go back to the Example \ref{12}. In this example, for the extended hyperplane $A^2_q$ and space $\rho$-$DerA^2_q$, assuming that $\Omega=dy\wedge dx$ and $|f| =(f_1,f_2)$, the Hamiltonian $\phi$-$\rho$-derivation associated to $f\in A^2_q$ has the following expression
\[X_f=q^{1-f_1}\mu_{A^2_q}(\frac{\partial}{\partial y})\cdot f\frac{\partial}{\partial x}-q^{f_2}\mu_{A^2_q}(\frac{\partial}{\partial x})\cdot f\frac{\partial}{\partial y}.\]
So, the Poisson bracket corresponding to $\phi_{A^2_q}$ gives as follows
\[\{f,g\}_{\rho}=-\rho(\Omega,g)\mu_{A^2_q}(X_f)\cdot g =-\rho(x+y,g)\{\q^{1-f_1}(\mu_{A^2_q}(\frac{\partial}{\partial y})\cdot f)(\mu_{A^2_q}(\frac{\partial}{\partial x})\cdot g)-q^{f_2}(\mu_{A^2_q}(\frac{\partial}{\partial x})\cdot f)(\mu_{A^2_q}(\frac{\partial}{\partial y})\cdot g)\}.\]
\end{example}
\begin{example}
This example is intended to give us a Poisson structure on quaternion algebra $H$ (Example \ref{13}). Let us to define the Poisson bracket $\{\cdot,\cdot\}_{\rho}$ on $H$ by the following structure
$$\{i,i\}_{\rho}=0,\; \{j,j\}_{\rho}=0,\; \{k,k\}_{\rho}=0,\;\{i,j\}_{\rho}=-\{j,i\}_{\rho}=k,\; \{k,i\}_{\rho}=-\{i,k\}_{\rho}=j,\; \{j,k\}_{\rho}=-\{k,j\}_{\rho}=i.$$
So, $(H,\{\cdot,\cdot\}_{\rho},\phi_H)$ is a Poisson Hom-$\rho$-commutative algebra.
\end{example}


\end{document}